\documentclass[reqno, 12pt]{amsart} 

\usepackage{amsmath, amssymb} 
\usepackage{amscd}            
\usepackage{amsxtra}          
\usepackage{upref}            

\usepackage[mathscr]{eucal}

\usepackage{mathtools}
\usepackage{verbatim}

\theoremstyle{plain}

\newtheorem{theorem}{Theorem}
\newtheorem{proposition}[theorem]{Proposition}
\newtheorem{lemma}[theorem]{Lemma}
\newtheorem{corollary}[theorem]{Corollary}

\theoremstyle{remark}
\newtheorem{remark}[theorem]{Remark}

\numberwithin{equation}{section}
\numberwithin{theorem}{section}
\newcommand{\be}%
  {\protect\setcounter{equation}{\value{subsubsection}}}  
\newcommand{\ee}%
  {\protect\setcounter{subsubsection}{\value{equation}}}

\DeclareMathAlphabet\BOONDOX{U}{rsfso}{m}{n}

\newcommand{\Q}{{\mathbb Q}}
\newcommand{\R}{{\mathbb R}}
\newcommand{\C}{{\mathbb C}}
\newcommand{\sel}{\BOONDOX{S}}

\newcommand{\esel}{{\sel}^{\#}}

\newcommand{\aselp}{{\mathfrak A}}
\newcommand{\asel}{{\aselp}^{\#}}
\newcommand{\bselp}{\BOONDOX{L}}

\newcommand{\aq}{{\mathbb A}_{\mathbb Q}}
\newcommand{\ak}{{\mathbb A}_K}

\newcommand{\gln}{{\rm GL}_n({\mathbb A}_{\mathbb Q})}

\newcommand{\glnak}{{\rm GL}_n(\ak)}

\newcommand{\hktz}{H(\alpha,T)}

\newcommand{\re}{\text{Re}}

\begin{document}


\title{On the absolute convergence of automorphic Dirichlet series}
\author{Ravi Raghunathan}

\address{Department of Mathematics \\ 
         Indian Institute of Technology Bombay\\
         Mumbai,\enspace  400076\\ India}          
\email{ravir@math.iitb.ac.in}

\begin{abstract} Let $F(s)=\sum_{n=1}^{\infty}\frac{a_n}{n^s}$ be a Dirichlet series in the axiomatically defined class $\asel$. The class $\asel$ is known to contain the extended Selberg class $\esel$, as well as all the $L$-functions of automorphic forms on $GL_n/K$, where $K$ is a number field. Let $d$ be the degree of $F(s)$.
We show that $\sum_{n<X}|a_n|=\Omega(X^{\frac{1}{2}+\frac{1}{2d}})$, and hence, that the abscissa of absolute convergence of $\sigma_a$ of $F(s)$ must satisfy $\sigma_a\ge 1/2+1/2d$.

\end{abstract}
\vskip 0.2cm

\maketitle


\markboth{RAVI RAGHUNATHAN}{On the absolute convergence of automorphic Dirichlet series}


\section{Introduction}\label{introduction}
In \cite{Ragh20} we introduced a class of Dirichlet series $\asel$ which is known to contain a very large number of $L$-functions attached to automorphic forms and also (strictly) contains the extended Selberg class $\esel$
of Kaczorowski and Perelli defined in \cite{KaPe99}. Associated to each Dirichlet series $F(s)=\sum_{n=1}^{\infty}\frac{a_n}{n^s}$ in $\asel$ is a non-negative real number - its degree $d_F$. We denote the subset of Dirichlet series of degree $d$ in $\asel$ by $\asel_d$. We state the main results of this paper first, referring the reader to Section \ref{defns} for the precise definitions of the degree $d_F$ and other terms appearing below.
\begin{theorem}\label{mainthm} Let $F(s)$ be an element of 
$\asel_d$ with $d\ge 1$. Then,
\begin{equation}\label{absconv}
\sum_{n<X}|a_n|=\Omega(X^{\frac{1}{2}+\frac{1}{2d}}).
\end{equation}
In particular, the abscissa of absolute convergence $\sigma_a$ satisfies $\sigma_a\ge1/2+1/2d$.
\end{theorem}
The following corollary covers the cases of greatest interest.
\begin{corollary}\label{automorphiclfn}
Let $L(s,\pi)=\sum_{n=1}^{\infty}\frac{a_n}{n^s}$ be the standard $L$-function associated to a unitary automorphic representation $\pi$ of $\gln$, where $\aq$ denotes the ad\`eles over $\Q$. Then, 
\[
\sum_{n<X}|a_n|=\Omega(X^{\frac{1}{2}+\frac{1}{2n}}).
\]
\end{corollary}
Indeed, it is known that $L(s,\pi)\in \asel$ and that its degree is $n$, so the corollary follows immediately from
the theorem. 

The corollary would appear to be new even for the $L$-functions of Maass forms associated to higher level congruence subgroups (for which $n=2$). 
Theorem \ref{mainthm} was known previously for the extended Selberg class $\esel$ (see Corollary 2 of \cite{KaPe2005}). Elements in $\esel$ are required to satisfy the analogue of the Generalised Ramanujan Conjecture 
(GRC) at infinity which is equivalent to the Selberg Eigenvalue Conjecture
for Maass eigenforms. Since these conjectures are very far from being established, Theorem \ref{mainthm} and Corollary \ref{automorphiclfn} are not subsumed by the earlier results. 

In addition, we note that elements of $\asel$ may have (a finite number of) poles at arbitrary locations and satisfy a more general functional equation than those of $\esel$. {\it A priori} they may have an arbitrary abscissa of absolute convergence, in contrast to the requirement $\sigma_a(F)\le 1$ for elements of $\esel$. Many other $L$-functions are known to belong to $\asel$ (but are not known to belong to $\esel$). These include the exterior square, symmetric square and tensor product $L$-functions associated to (unitary) automorphic representations of $\glnak$, where $\ak$ denotes the ad\`eles over a number field $K$. The $L$-functions of half integral weight forms and Siegel modular forms also belong to $\asel$, but in general, do not belong to $\esel$. Theorem \ref{mainthm} thus applies to a substantially larger class of examples.

The proof of Theorem \ref{mainthm} uses a transform introduced by Soundararajan in \cite{Sound05} for the case $d=1$ in the context of the Selberg class $\sel$, but improves on the relevant stationary phase techniques following the arguments in \cite{BaRa20}. These allow us to prove
an asymptotic formula for the ``standard addtive twist''
\[ 
F(0,\alpha,1/d):=\sum_{T<n<4T}^{\infty}a_ne^{-id\alpha n^{1/d}}\sim c_0T^{1/2+1/2d}+o(T^{1/2+1/2d})
\]
for some constant $c_0$, when $\sigma_a<1/2+1/d-\delta$ for any $\delta>0$
(see equation \eqref{finalequality}), from which Theorem \ref{mainthm} follows easily. 

The asymptotic formula for the standard additive twist of elements in $\esel$ was proved in \cite{KaPe2005} {\it without any further assumptions}. The proof invokes the properties of Fox hypergeometric functions and other complex analytics techniques. While our method is unable to recover this more subtle statement, it does produce a completely different and shorter proof of Theorem \ref{mainthm} even for the class $\esel$.

In \cite{KaPe15} it is shown that the conclusion of Theorem \ref{mainthm} holds for series which are polynomials in the elements of the Selberg class $\sel$. It is likely that the same ideas work for polynomials in the elements of $\aselp$, the class of series in $\asel$ which have an Euler product. However, we do not attempt this here.

\section{Some basic definitions}\label{defns}

The class $\asel$ was defined in \cite{Ragh20} as follows. For $s\in \C$, we
write $s=\sigma+it$, where $\sigma,t\in \R$. Let $F(s)\ne 0$ be a meromorphic function on $\C$. 
We consider the following conditions on $F(s)$.
\begin{enumerate}
\item [(P1)] The function $F(s)$ is given by a
Dirichlet series $\sum_{n=1}^{\infty}\frac{a_n}{n^{s}}$ 
with abscissa of absolute convergence $\sigma_a\ge 1/2$.
\item[(P2)] There is a polynomial $P(s)$ such that $P(s)F(s)$ 
extends to an entire function, and such that given any
vertical strip $\sigma_1\le \sigma \le \sigma_2$, there is some $M\in\R$ such that $P(s)F(s)\ll (1+t)^M$.
\item[(P3)] There exist a real number $Q>0$, a complex number $\omega$ such that 
$\vert\omega\vert=1$, and a function $G(s)$ of the form
\begin{equation}\label{gammafactor}
G(s)=\prod_{j=1}^{r}\Gamma(\lambda_j s+\mu_j)\prod_{j'=1}^{r'}\Gamma(\lambda_{j^{\prime}}^{\prime}s+\mu_{j^{\prime}}^{\prime})^{-1},
\end{equation}
where $\lambda_j, \lambda_{j^{\prime}}^{\prime}>0$ are real numbers, $\mu_j, \mu_{j^{\prime}}^{\prime}\in \C$, 
and $\Gamma(s)$ denotes the usual gamma function, such that
\begin{equation}\label{fnaleqn}
\Phi(s):=Q^{s}G(s)F(s)=\omega\overline{\Phi(1-\bar{s})}.
\end{equation}
\end{enumerate}
We will denote by $\asel$ the set of (non-zero) meromorphic functions satisfying (P1)-(P3). We set 
$d_F=2\sum_{j=1}^{r}\lambda_j-2\sum_{j'=1}^{r^{\prime}}\lambda_{j^{\prime}}^{\prime}$. Theorem 2.1 of \cite{Ragh20} shows that $d_F$ does not depend on the choice of the functions $G(s)$ that appear in \eqref{fnaleqn}. The number $d_F$ is called the 
degree of the function $F(s)$. The set of all functions $F(s)\in \asel$ with $d_F=d$ will be denoted by $\asel_d$.

The class $\esel$ is defined as the set of series $F(s)\in \asel$ satisfying the conditions $\sigma_a\le 1$, $P(s)=(s-1)^m$ for some $m\ge 0$,
$r^{\prime}=0$, and $\mu_j\ge 0$ for all $1\le j\le r$, define. As we have outlined in the introduction, 
even when we expect a 
series $F(s)$ to belong to $\esel$, we can only rarely prove that this is the case, since major conjectures like the GRC at infinity are involved. In addition, there are a large number of examples that belong to $\asel$, but do not belong to $\esel$. Two simple examples to keep in mind are $\zeta(2s-1/2)$ and $\zeta(s+1/2)\zeta(s-1/2)$.

More detailed rationales for working in $\asel$ rather than in $\esel$, or in the class $\bselp$ introduced by A. Booker in \cite{Booker2015}, may be found in \cite{Ragh20} and \cite{BaRa20}.

\section{Preliminaries}
In this section we record a few facts from \cite{BaRa20} which we will need 
for our proof. We first fix the following notation. For a complex function 
$f(s)$ we define $\tilde{f}(s)=\overline{f(\bar{s})}$.

Let $z=x+iy$, and assume that $-\pi+\theta_0<\arg(z+it)<\pi-\theta_0$ for some 
$\theta_0>0$. From Section 2.2 of \cite{BaRa20} (see equations (2.1)-(2.4) of that paper), we retrieve
\begin{equation}\label{fourthgammaest}
\frac{\tilde{G}(1-x-it)}{G(x+it)}
= (Ce^{-d}t^{d})^{(\frac{1}{2}-x)} e^{-itd\log\frac{t}{e}}t^{iA}e^{iB}C^{-it}\cdot 
(1+O(1/t)),
\end{equation}
where 
\[
A=-i((\bar{\mu}-\mu)-(\bar{\mu^{\prime}}-\mu^{\prime})),\quad C=\prod_{j,j'=1}^{r,r'}{\lambda_j}^{2\lambda_j}
{\lambda_{j^{\prime}}^{\prime}}^{-2\lambda_{j^{\prime}}^{\prime}}, 
\]
and
\begin{flalign}
B=&-i\left(\sum_{j=1}^r(\bar{\mu}_j-\mu_j)\log\lambda_j-\sum_{j'=1}^r(\overline{\mu_{j^{\prime}}^{\prime}}-\mu_{j^{\prime}}^{\prime})\log\lambda_{j,}\right)\nonumber\\
&-(\mu-\bar{\mu})+(\mu^{\prime}-\bar{\mu^{\prime}})-((\mu-\bar{\mu})-(\mu^{\prime}-\bar{\mu^{\prime}})+d/2)\frac{\pi}{2},
\end{flalign}
with
\[
\mu=\sum_{j=1}^{r}\mu_j\quad\text{and}\quad \mu^{\prime}=\sum_{j^{\prime}=1}^{r}\mu_{j^{\prime}}^{\prime}.
\]
Note that $A\in \R$ and $C>0$. Replacing $x+it$ by $x+it+w$ ($w=u+iv$) in \eqref{fourthgammaest}, 
and taking absolute values, we obtain
\begin{equation}\label{zerothgammaest}
\frac{\tilde{G}(1-x-it-w)}{G(x+it+w)}\ll (1+|t+v|)^{-d(x-1/2+u)}.
\end{equation}

Additionally, we will need the following lemma from \cite{BaRa20} which allows
us to pass from $F(s)$ to an everywhere convergent Dirichlet series.
\begin{lemma}\label{basiclemma} Let $w=u+iv$, $z=x+iy$, $p>0$ and $d>0$.
If $F(s)\in \asel_d$ is holomorphic at $s=z+it$ and $0<\eta<1-x+p-\sigma_a$, we have
\begin{equation}\label{basiclemmaeqn}
F(z+it)=\sum_{n=1}^{\infty}\frac{a_ne^{-(n/X)^{p}}}{n^{z+it}}
+r_1(t,X)+r_2(t,X),
\end{equation}
where $r_1(t,X):=O(X^{\sigma_a-x}e^{-|t|/p})$ 
is identically zero if $F(z)$ is entire, and 
\begin{equation}\label{rtwoz}
r_2(t,X):=\frac{1}{2\pi i p}\int_{u=-p+\eta}F(z+it+w)X^{w}\Gamma(w/p) dw
\ll O(t^{d(\frac{1}{2}+p-x-\eta)}X^{-p+\eta}),
\end{equation}
where $u=-p+\eta$ is a line on which $F(z+it+w)$ is holomorphic.
\end{lemma}
\begin{remark} We can apply the lemma above to $\tilde{F}(s)$ instead of $F(s)$. This yields
\begin{equation}\label{secondbasiclemmaeqn}
\tilde{F}(1-z-it)=\sum_{n=1}^{\infty}\frac{\overline{a_n}e^{-(n/X)^{p}}}{n^{1-z-it}}
+\tilde{r}_1(t,X)+\tilde{r}_2(t,X),
\end{equation}
where $\tilde{r}_i(t,X)$ satisfies the same estimates as $r_i(t,X)$ when $x$ is replaced by $1-x$, for $i=1,2$.
\end{remark}
\begin{remark}
It has been pointed out to me by D. Surya Ramana that the lemma above is valid for 
$0<\eta<p$ if we use the standard convexity bounds for $F(s)$. In this paper 
we will need only the weaker statement made in the lemma.
\end{remark}

\section{Soundararajan's transform}

Suppose that $F(s)\in \asel_d$, with $d\ge 1$. 
For $\alpha\ge 1$ and $T$ chosen large enough that 
$F(1/2+it)$ is holomorphic for $t\ge T$, we define
\begin{equation}\label{hkdefn}
\hktz:=\frac{1}{\sqrt{\alpha}}
\int_{K_T}
F(1/2+it)e^{idt\log\left[\frac{t}{e\alpha}\right]-i\frac{\pi}{4}}dt,
\end{equation}
where $K_T=[2\alpha T,3\alpha T]$. Soundararajan introduced (a mild variant of) this transform for $d=1$ and $F\in \sel$ in \cite{Sound05}, and we used a similar transform in \cite{BaRa20} to study $\asel_d$ when $1<d<2$.
In what follows, $\alpha$ will be fixed, so we will study the behaviour of $\hktz$ as a function of $T$.

We use Lemma \ref{basiclemma} when $z=1/2$.
Substituting for $F(1/2+it)$ from equation \eqref{basiclemmaeqn}, we obtain (for any $X_1>0$),
\begin{equation}\label{bltransform}
\hktz=\frac{1}{\sqrt{\alpha}}\sum_{n=1}^{\infty}\frac{a_n}{\sqrt{n}}e^{-(n/X_1)^{p}}I_n
+R_1(\alpha,T,X_1)+R_2(\alpha,T,X_1),
\end{equation}
where $R_i(\alpha,T,X_1)=\frac{1}{\sqrt{\alpha}}\int_{K_T}r_i(t,X_1)e^{idt\log\left[\frac{t}{e\alpha}\right]-i\frac{\pi}{4}}dt$,
for $i=1,2$, 
\begin{equation}\label{indefn}
I_n=I_n(\alpha,T):=\frac{1}{2\pi i} \int_{K_T}e^{idt\log\left[\frac{t}{e\alpha x_n}\right]
-i\frac{\pi}{4}}dt,
\end{equation}
and $x_n=n^{1/d}$. Using the estimates for $r_1(t,X_1)$ given above, we see that
$R_1(\alpha,T,X_1)=O(X_1^{\sigma_a-1/2}e^{-\alpha T})$. 
We will be choosing $X_1=T^{d+\rho}$ for some $\rho>0$. The term $R_1(\alpha,T,X_1)$ will 
thus have exponential decay in $T$ since $\alpha$ is fixed. Thus we can assume $R_1(\alpha,T,X_1)=O(1)$.

We estimate the term $R_2(\alpha,T,X_1)$ trivially. Indeed, integrating the absolute value of the integrand and using the estimate \eqref{rtwoz}, produces
\[
R_2(\alpha,T,X_1)=O(T^{d(p+1-\eta)}X_1^{-p+\eta}).
\]
Since $\rho>0$, if $p-\eta$ is chosen large enough,  $R_2(\alpha,T,X_1)=O(1)$. 

We record this as a proposition.
\begin{proposition}\label{rprop}
With notation as above, $X_1=T^{d+\rho}$, and for $p-\eta$ chosen large enough, 
\[
R_i(\alpha,T,X_1)=O(1).
\]
for $i=1,2$.
\end{proposition}
It remains to evaluate the sum 
appearing in \eqref{bltransform} which we will do in the next section.

\section{Estimating the oscillatory integral $I_n$}
We will require two lemmas for evaluating the oscillatory integrals $I_n$ that appear in equation \eqref{bltransform}. The first is well known, and can be found in Section 1.2 of Chapter VIII in \cite{Stein93}, for instance. It is needed to estimate $I_n$ when $n$ is relatively small or large compared to $T^d$.
\begin{lemma}\label{largen} Suppose that $g(t)$ is a function of bounded variation on 
an interval $K=[a,b]$ and $|g(t)|\le M$ for all $t\in K$. For any ${\mathcal C}^1$-function $f$ on $K$,
if $f^{\prime}(t)$ is monotonic and $|f^{\prime}(t)|\ge m_1$ on $K$,
\[
\int_K g(t)e^{if(t)}dt \ll \frac{1}{m_1}\left\{\vert M\vert +\int_K\vert g^{\prime}(t)\vert dt\right\}.
\]
\end{lemma}
To evaluate the integrals $I_n$ when $n$ has roughly the size $T^d$ we need
Lemma 3.3 of \cite{BoBo04}.
\begin{lemma}\label{midn}
Suppose that $f$ is a ${\mathscr C}^3$-function on an interval 
$K=[a,b]$ and $f^{\prime\prime}(t)\ne 0$ on $K$. 
If $f^{\prime}(c)=0$ for some $c\in K$, and $m>0$ is such 
that  $|f^{\prime\prime\prime}(t)|\le m$ for 
$t\in K\cap \left[c-\left\vert\frac{f^{\prime\prime}(c)}{m}\right\vert,c+\left\vert\frac{f^{\prime\prime}(c)}{m}\right\vert\right]$,
then
\[
\int_K e^{if(t)}dt=e^{\pm i\frac{\pi}{4}}\frac{e^{if(c)}}{\sqrt{|f^{\prime\prime}(c)|}}
+O\left(\frac{m}{|f^{\prime\prime}(c)|^2}\right)
+O\left(\frac{1}{|f^{\prime}(a)|}+\frac{1}{|f^{\prime}(b)|}\right).
\]
The $\pm$ in the expression above occurs according to the sign of $f^{\prime\prime}(c)$.
\end{lemma}
We recall that $K_T=[2\alpha T,3\alpha T]$. In the notation of the lemmas above,
\[
I_n=\int_{K}g(t)e^{if(t)}dt,
\]
where $K=K_T$, $g(t)\equiv 1$ and
\[
f(t)=dt\log\left[\frac{t}{e\alpha n^{\frac{1}{d}}}\right]-\frac{i\pi}{4}.
\]
\begin{proposition}\label{inest}
For $n\le T^d$ and $4T^d\le n<T^{d+\rho}$, 
\begin{equation}\label{insmalllarge}
I_n=O(1).
\end{equation}
If $T^d<n<4T^d$, 
\begin{equation}\label{inmid}
I_n=\sqrt{\alpha}d^{-\frac{1}{2}}n^{1/2d}e^{id\alpha n^{1/d}}+O(1).
\end{equation}
\end{proposition}
\begin{proof} 
We follow the proof (for $d=1$) in \cite{BaRa20}. Indeed, we have
\[
f^{\prime}(t)=d\log\left[\frac{t}{\alpha n^{\frac{1}{d}}}\right],\,\,f^{\prime\prime}(t)=d/t\,\,\text{and}\,\, 
f^{\prime\prime\prime}(t)=-d/t^2.
\]
If $n\le T^d$, then $|f^{\prime}(t)|\ge d\log 2$. Similarly,
if $4T^d\le n<T^{d+\rho}$, $|f^{\prime}(t)|\ge d\log 4/3$.
Then Lemma \ref{largen} shows that $I_n=O(1)$, and \eqref{insmalllarge} follows.

If $T^d<n<4T^d$, we proceed as follows. Note that $f^{\prime}(c)=0$ 
means that $c=\alpha n^{1/d}$. The first term on the right in Lemma \ref{midn} thus yields $\sqrt{\alpha}d^{-\frac{1}{2}}n^{1/2d}e^{-id\alpha n^{1/d}}$. Now choose $m=3d/c^2$,
so $f^{\prime\prime}(c)/m=c/3$. If $t\in K_T\cap [2c/3,4c/3]$,
$|f^{\prime\prime\prime}(t)|=9d/4c^2\le m$. Thus, the hypotheses of Lemma \ref{midn} are satisfied. The first error term in the lemma yields
\[
O\left(\frac{m}{|f^{\prime\prime}(c)|^2}\right)=O(1),
\]
while the last two error terms also yield $O(1)$. This proves \eqref{inmid}.
\end{proof}

Note that when estimating the sum in equation \eqref{bltransform}, it is enough to estimate the sum for 
$n<T^{d+\rho}$, since the terms in the sum decay exponentially when $n$ exceeds this. Using the 
the estimates \eqref{insmalllarge} and \eqref{inmid} in the sum in equation \eqref{bltransform} 
yields (for $X_1\ge T^{d+\rho}$)
\begin{equation}\label{interimeqn}
\frac{1}{\sqrt{\alpha}}\sum_{n=1}^{\infty}\frac{a_n}{\sqrt{n}}e^{-(n/X_1)^{p}}I_n
=\sum_{T^d<n<4T^d}\frac{a_n}{\sqrt{n}}e^{-(n/X_1)^{p}}d^{-\frac{1}{2}}n^{1/2d}e^{-id\alpha n^{1/d}}
+O(X_1^{\sigma_a-\frac{1}{2}+\varepsilon}),
\end{equation}
for any $\varepsilon>0$.
Combining equation \eqref{interimeqn} with the estimates in Proposition \ref{rprop}, we get the following proposition.
\begin{proposition}\label{firstestimate}
For any $\varepsilon>0$, we have
\begin{equation}\label{hkfirsteqn}
\hktz=\sum_{T^d<n<4T^d}\frac{a_n}{\sqrt{n}}d^{-\frac{1}{2}}n^{1/2d}e^{-(n/T^{d+\rho})^{p}}e^{-id\alpha n^{1/d}}
+O(T^{(d+\rho)(\sigma_a-\frac{1}{2}+\varepsilon)}).
\end{equation}
\end{proposition}

\section{A second estimate for $\hktz$}
We now evaluate the transform $\hktz$ in a second way, hewing to Soundararajan's arguments in \cite{Sound05} for $d=1$. Applying the functional equation to the integrand, and then using equation 
\eqref{fourthgammaest} for $\tilde{F}(1/2-it)$, gives
\[
\hktz=\frac{\omega e^{iB}}{\sqrt{\alpha}}\int_{K_T}
\tilde{F_2}(1/2-it)(CQ^2\alpha^d)^{-it}t^{iA}\left[1+O(1/t)\right]dt.
\]
Using equation \eqref{secondbasiclemmaeqn}, we obtain
\begin{flalign}\label{postfehk}
\hktz=\frac{\omega e^{iB}}{\sqrt{\alpha}}\int_{K_T}&
\left[\sum_{n=1}^{\infty}\frac{\overline{a_n}}{\sqrt{n}}e^{-(n/X_2)^p}
+\tilde{r}_1(t,X_2)+\tilde{r}_2(t,X_2)\right]\nonumber\\
&\times(n^{-1}CQ^2\alpha^d)^{-it}t^{iA}\left[1+O(1/t)\right]dt.
\end{flalign}
for $X_2>0$. Imitating the arguments used for majorising $R_i(\alpha,T,X_1)$ in Proposition
\ref{rprop}, the terms 
$\tilde{R}_i(\alpha,T,X_2)=\int_{K_T}\tilde{r}_i(t,X_2)(CQ^2\alpha^d)^{-it}t^{iA}dt$ 
yield $O(1)$ when estimated trivially, if $X_2=T^{d+\rho}$ for some $\rho>0$,
and $p-\eta$ is chosen large enough. We also have
\[
\int_{K_T}\tilde{r}_i(t,X_2)O(1/t)dt \ll \tilde{R}_i(\alpha,T,X_2)=O(1)
\]
for $i=1,2$.

We switch the order of summation and integration in the first term of \eqref{postfehk}  to get 
the expression
\begin{equation}\label{suminpostfehk}
\frac{\omega e^{iB}}{\sqrt{\alpha}}\sum_{n=1}^{\infty}\frac{\overline{a_n}}{\sqrt{n}}e^{-(n/X_2)^p}J_n,
\end{equation}
where
\[
J_n=\int_{K_T}(n^{-1}C\pi Q^2\alpha)^{-it}t^{iA}dt.
\]
As before, it is enough to evaluate or estimate this sum when $n<X_2^{1+\varepsilon}$.
Choose $m$ such $a_m\ne 0$, and fix $\alpha$ so that $m=CQ^2\alpha^d$. Then
$J_m$ can be evaluated exactly to give
\[
J_m=\frac{(3^{1+iA}-2^{1+iA})}{1+iA}\cdot\alpha^{1+iA}T^{1+iA}.
\]
For $n\ne m$, we use integration by parts to estimate $J_n$. We have
\[
J_n=O(1/\log(n^{-1}CQ^2\alpha^d)=O(1)
\]
Let 
\[
\kappa=\omega e^{iB}\sqrt{C}Q\alpha^{\frac{1-d}{2}+iA}\frac{(3^{1+iA}-2^{1+iA})}{1+iA}
\]
Substituting for $J_n$ in \eqref{suminpostfehk}, we obtain (for any $\varepsilon>0$)
\[
\frac{\omega e^{iB}}{\sqrt{\alpha}}\sum_{n=1}^{\infty}\frac{\overline{a_n}}{\sqrt{n}}e^{-(n/T^{d+\rho})^p}J_n
=\kappa a_mT^{1+iA}+O(T^{(d+\rho)(\sigma_a-\frac{1}{2}+\varepsilon)})
\]
when $X_2=T^{d+\rho}$. The sum involving 
$\int_{K_T}(n^{-1}C\pi Q^2\alpha)^{-it}t^{iA}O(1/t)dt$ is dominated by the sum involving $J_n$ above.
We consolidate the arguments in this section as
\begin{proposition}\label{hksecond}
Suppose that  $a_m\ne 0$ and $\alpha$ is chosen so that $m=CQ^2\alpha^d$. Then,
\begin{equation}\label{hksecondeqn}
\hktz=\kappa a_mT^{1+iA}+O(T^{(d+\rho)(\sigma_a-\frac{1}{2}+\varepsilon)}).
\end{equation}
\end{proposition}

\section{The proof of Theorem \ref{mainthm}}\label{last}
We now have all the estimates necessary to prove Theorem \ref{mainthm}.
Equating \eqref{hkfirsteqn} and \eqref{hksecondeqn} gives us  
\[
\sum_{T^d<n<4T^d}\frac{a_n}{\sqrt{n}}n^{1/2d}e^{-(n/T^{d+\rho})^{p}}e^{-id\alpha n^{1/d}}
=\kappa d^{\frac{1}{2}}a_mT^{1+iA}+O(T^{(d+\rho)(\sigma_a-\frac{1}{2}+\varepsilon)}).
\]
This can be rewritten as
\begin{equation}\label{finalequality}
\sum_{T<n<4T}a_ne^{-(n/T^{d+\rho})^{p}}e^{-id\alpha n^{1/d}}
=\kappa d^{\frac{1}{2}}a_mT^{\frac{1}{2}+\frac{1}{2d}+iA}+O(T^{(1+\frac{\rho}{d})(\sigma_a-\frac{1}{2}+\varepsilon)
+\frac{1}{2}-\frac{1}{2d}}).
\end{equation}
Suppose that $F(s)$ converges absolutely when $\re(s)= 1/2+1/2d$, so $\sigma_a\le 1/2+1/2d$. If $\rho$ and $\varepsilon$ are chosen small enough, we see that the second term on the right hand side of \eqref{finalequality} is actually $o(T^{\frac{1}{2}+\frac{1}{2d}})$. But then, for $T$ large enough,
\[
\sum_{T<n<4T}|a_n|\ge \left\vert\sum_{T<n<4T}a_ne^{-(n/T^{(d+\rho)})^{p}}e^{-id\alpha n^{1/d}}\right\vert\ge 2^{-1}|\kappa d^{\frac{1}{2}}a_m|T^{\frac{1}{2}+\frac{1}{2d}}.
\]
This contradicts the assumption that $F(s)$ converges absolutely when $\re(s)=1/2+1/2d$,
and Theorem \ref{mainthm} follows.

\bibliographystyle{alpha}
\bibliography{../../../../../Bibtex/master2020}

\begin{thebibliography}{{Rag}20}

\bibitem[BB04]{BoBo04}
E.~Bombieri and J.~Bourgain.
\newblock A remark on {B}ohr's inequality.
\newblock {\em Int. Math. Res. Not.}, (80):4307--4330, 2004.

\bibitem[Boo15]{Booker2015}
Andrew~R. Booker.
\newblock {$L$}-functions as distributions.
\newblock {\em Math. Ann.}, 363(1-2):423--454, 2015.

\bibitem[BR20]{BaRa20}
R~Balasubramanian and Ravi Raghunathan.
\newblock Beyond the extended {S}elberg class: {$1<d<2$}.
\newblock {\em arXiv:2011.07525}, 2020.
\newblock Submitted for publication.

\bibitem[KP99]{KaPe99}
Jerzy Kaczorowski and Alberto Perelli.
\newblock On the structure of the {S}elberg class. {I}. {$0\leq d\leq 1$}.
\newblock {\em Acta Math.}, 182(2):207--241, 1999.

\bibitem[KP05]{KaPe2005}
J.~Kaczorowski and A.~Perelli.
\newblock On the structure of the {S}elberg class. {VI}. {N}on-linear twists.
\newblock {\em Acta Arith.}, 116(4):315--341, 2005.

\bibitem[KP15]{KaPe15}
Jerzy Kaczorowski and Alberto Perelli.
\newblock General {$\Omega$}-theorems for coefficients of {$L$}-functions.
\newblock {\em Proc. Amer. Math. Soc.}, 143(12):5139--5145, 2015.

\bibitem[{Rag}20]{Ragh20}
Ravi {Raghunathan}.
\newblock Beyond the extended {S}elberg class: {$d_F\le 1$}.
\newblock {\em {arXiv:2005.11381}}, 2020.
\newblock Submitted for publication.

\bibitem[Sou05]{Sound05}
K.~Soundararajan.
\newblock Degree 1 elements of the {S}elberg class.
\newblock {\em Expo. Math.}, 23(1):65--70, 2005.

\bibitem[Ste93]{Stein93}
Elias~M. Stein.
\newblock {\em Harmonic analysis: real-variable methods, orthogonality, and
  oscillatory integrals}, volume~43 of {\em Princeton Mathematical Series}.
\newblock Princeton University Press, Princeton, NJ, 1993.
\newblock With the assistance of Timothy S. Murphy, Monographs in Harmonic
  Analysis, III.

\end{thebibliography}

\end{document}